\documentclass[12pt]{article}
\usepackage{amssymb}
\usepackage{amsthm}
\usepackage{amsmath}
\usepackage{amscd}
\usepackage{array}
\usepackage{url}
\usepackage[latin2]{inputenc}
\usepackage{t1enc}
\usepackage{enumerate}
\usepackage[mathscr]{eucal}
\usepackage[british]{babel}
\usepackage[dvips]{graphicx}
\usepackage[dvips]{epsfig}
\usepackage{epsf}
\usepackage{indentfirst}
\usepackage[all]{xy}

\newcommand{\Ker}{\mathrm{Ker}}

\newcommand{\Hom}{\mathrm{Hom}}
\newcommand{\ind}{\mathrm{ind}}
\newcommand{\Tor}{\mathrm{Tor}}

\newcommand{\Stab}{\mathrm{Stab}}

\newcommand{\GL}{\mathrm{GL}}

\newcommand{\id}{\mathrm{id}}

\newcommand{\Sub}{\mathrm{Sub}}

\newcommand{\bg}{(\hspace{-0.06cm}(}
\newcommand{\jg}{)\hspace{-0.06cm})}

\newcommand{\bs}{[\hspace{-0.04cm}[}
\newcommand{\js}{]\hspace{-0.04cm}]}
\newtheorem{thm}{Theorem}
\newtheorem{pro}[thm]{Proposition}
\newtheorem{lem}[thm]{Lemma}
\newtheorem{cor}[thm]{Corollary}

\theoremstyle{definition}
\newtheorem*{rem}{Remark}

\begin{document}
\title{Exactness of the reduction on \'etale modules}
\author{Gergely Z\'abr\'adi}
\date{\today}
\maketitle

\begin{abstract}
We prove the exactness of the reduction map from \'etale $(\varphi,\Gamma)$-modules over completed localized group rings of compact open subgroups of unipotent $p$-adic algebraic groups to usual \'etale $(\varphi,\Gamma)$-modules over Fontaine's ring. This reduction map is a component of a functor from smooth $p$-power torsion representations of $p$-adic reductive groups (or more generally of Borel subgroups of these) to $(\varphi,\Gamma)$-modules. Therefore this gives evidence for this functor---which is intended as some kind of $p$-adic Langlands correspondence for reductive groups---to be exact. We also show that the corresponding higher $\Tor$-functors vanish. Moreover, we give the example of the Steinberg representation as an illustration and show that it is acyclic for this functor to $(\varphi,\Gamma)$-modules whenever our reductive group is $\GL_{d+1}(\mathbb{Q}_p)$ for some $d\geq 1$.
\end{abstract}

\section{Introduction}

\subsection{Colmez' work}
In recent years it has become increasingly clear that some kind of $p$-adic version of the local Langlands conjectures should exist. However, a precise formulation is still missing. It is all the more remarkable that Colmez has recently managed to establish such a correspondence between $2$-dimensional $p$-adic Galois representations of $\mathbb{Q}_p$ and continuous irreducible unitary $p$-adic representations of $\GL_2(\mathbb{Q}_p)$. In fact, Colmez \cite{Co1,Co2} constructed a functor from smooth torsion $P$-representations to \'etale $(\varphi,\Gamma)$-modules where $P$ is the standard parabolic subgroup of $\GL_2(\mathbb{Q}_p)$. Whenever we are given a unitary $\GL_2(\mathbb{Q}_p)$-representation $V$, we may find a $\GL_2(\mathbb{Q}_p)$-invariant lattice $L$ inside it. Hence we can take the restriction to $P$ of the reduction $L/p^mL$ mod $p^m$ for some positive integer $m$ and pass to $(\varphi,\Gamma)$-modules using Colmez' functor. The $(\varphi,\Gamma)$-module corresponding to the initial representation of $\GL_2(\mathbb{Q}_p)$ will be the projective limit of these $(\varphi,\Gamma)$-modules when $m$ tends to infinity. The miracle is that whenever we started with an irreducible supercuspidal $\GL_2$-representation in characteristic $p$ the resulting $(\varphi,\Gamma)$-module will be $2$-dimensional and hence correspond to a $2$-dimensional modulo $p$ Galois representation of the field $\mathbb{Q}_p$. The image of $1$-dimensional and pricipal series representations is, however, $0$ and $1$ dimensional, respectively (see Thm.\ 10.7 in \cite{Vi2} and Thm.\ 10.4 in \cite{P}).

\subsection{The Schneider-Vigneras functors}
Even more recently, Schneider and Vigneras \cite{SV} managed to generalize
Colmez' functor to general $p$-adic reductive groups. Their context is the
following. Let $G$ be the group of $\mathbb{Q}_p$-points of a
$\mathbb{Q}_p$-split connected reductive group over $\mathbb{Q}_p$ whose
centre is also assumed to be connected. To review their construction we fix a
Borel subgroup $P=TN$ with split torus $T$ and unipotent radical $N$. We also
fix an appropriate compact open subgroup $N_0$ which gives rise to the
`dominant' submonoid $T_+:=\{t\in T\mid tN_0t^{-1}\subseteq N_0\}$ in $T$. On
the one side we consider the abelian category $\mathcal{M}_{o-tor}(P)$ of all
smooth $o$-torsion representations of the group $P$ where $o$ is the ring of
integers in a fixed finite extension $K/\mathbb{Q}_p$. On the other side a
monoid ring $\Lambda(P_+)$ is introduced for the monoid $P_+:=N_0T_+$ and we
denote the category of all (left unital) $\Lambda(P_+)$-modules by
$\mathcal{M}(\Lambda(P_+))$. Such a module $M$ is called \'etale if every
$t\in T_+$ acts, informally speaking, with slope zero on $M$. The universal
$\delta$-functor $V\mapsto D^i(V)$ for $i\geq 0$ from $\mathcal{M}_{o-tor}(P)$
to the category $\mathcal{M}_{et}(\Lambda(P_+))$ of \'etale
$\Lambda(P_+)$-modules is constructed the following way. $D^i$ are the derived
functors of a contravariant functor $D\colon \mathcal{M}_{o-tor}(P)\rightarrow
\mathcal{M}(\Lambda(P_+))$ which is not exact in the middle, but takes
surjective, resp.\ injective maps to injective, reps.\ surjective maps. (Hence
$D\subsetneq D^0$ in general.) Note that the $\Lambda(P_+)$-module $D(V)$ is
not \'etale in general, but $D^i(V)$ always lies in
$\mathcal{M}_{et}(\Lambda(P_+))$ for any $i\geq 0$. The modules $D^i(V)$ are not expected to have good properties in general. This is why it is natural to pass to some topological localization $\Lambda_{\ell}(P_{\star})$ of the group ring $\Lambda(P_{\star})$ of a submonoid $P_{\star}$ of $P_+$ generated by $P_0$, $\varphi$, and $\Gamma$. The corresponding abelian category $\mathcal{M}_{et}(\Lambda_{\ell}(P_{\star}))$ of \'etale $\Lambda_{\ell}(P_{\star})$-modules is a generalization of Fontaine's $(\varphi,\Gamma)$-modules. Indeed, whenever $G=\GL_2(\mathbb{Q}_p)$ (in this case we denote by $S_{\star}$ the standard monoid inside $\GL_2(\mathbb{Q}_p)$ and note that $N_0\cong\mathbb{Z}_p$) then the objects that are finitely generated over the smaller localized ring $\Lambda_{\ell}(N_0)\cong\Lambda_{F}(\mathbb{Z}_p)$ are exactly Fontaine's $(\varphi,\Gamma)$-modules. This construction leads to the universal $\delta$-functor $D^i_{\ell}(V)$. The fundamental open question in \cite{SV} is for which class of $P$-representations $V$ are the modules $D^i_{\ell}(V)$ finitely generated over $\Lambda_{\ell}(N_0)$. Moreover, with the help of a Whittaker type functional $\ell$ one may pass to the category $\mathcal{M}_{et}(\Lambda_F(S_{\star}))$ for the standard monoid $S_{\star}$ in $\GL_2(\mathbb{Q}_p)$. This way one obtains a $\delta$-functor $D^i_{\Lambda_F(S_{\star})}$ from $\mathcal{M}_{o-tor}(P)$ to the category of not necessarily finitely generated $(\varphi,\Gamma)$-modules \`a la Fontaine. For the group $G=\GL_2(\mathbb{Q}_p)$ Colmez' original functor coincides with $D^0_{\Lambda_F(S_{\star})}$ and the higher $D^i_{\Lambda_F(S_{\star})}$ vanish.

\subsection{Outline of the paper}
The aim of this short note is to investigate the exactness properties of the functors constructed by Schneider and Vigneras \cite{SV}. Whenever $G\neq\GL_2(\mathbb{Q}_p)$ then the reduction map
\begin{equation*}
\ell\colon\Lambda_{\ell}(N_0)\twoheadrightarrow\Lambda_F(\mathbb{Z}_p)
\end{equation*}
has a nontrivial kernel and hence is not flat since $\Lambda_{\ell}(N_0)$ has
no zero divisors and therefore any flat module is torsion-free. However, the extra
\'etale $\varphi$-structure allows us to show that the reduction
functor from \'etale $\varphi$-modules over $\Lambda_{\ell}(N_0)$
to \'etale $\varphi$-modules over $\Lambda_F(\mathbb{Z}_p)$
induced by $\ell$ is still exact if we restrict ourselves to
\emph{pseudocompact} $\Lambda_{\ell}(N_0)$-modules which includes
those finitely generated. The proof relies on the existence of a descending
separated filtration of (two-sided) ideals $J_n$ of $\Lambda_{\ell}(N_0)$ such that
$\varphi(J_n)\subseteq J_{n+1}$. In
section \ref{30} we use this to show that in fact the higher
$\Tor$-functors
$\Tor^i_{\Lambda_{\ell}(N_0)}(\Lambda_F(\mathbb{Z}_p),M)$ vanish
for $i\geq 1$ whenever $M$ is a pseudocompact \'etale
$\varphi$-module over $\Lambda_{\ell}(N_0)$.

In section \ref{31} we investigate the example of the Steinberg
representation $V_{St}$. We show that in this case we have
$D^0(V_{St})=D(V_{St})$ and, in particular $D^0(V_{St})$ is
finitely generated over $\Lambda_{\ell}(N_0)$. Moreover, we prove
that all the higher $D^i(V_{St})$ vanish for $i\geq 1$. This is
the first known example of a smooth $o$-torsion $P$-representation
with finitely generated $D^0_{\ell}$ and with known $D^i_{\ell}$
for all $i\geq0$ apart from those for $\GL_2(\mathbb{Q}_p)$. Hence $V_{St}$ is acyclic for the functor
$D_{\ell}$ and also for the functor $D_{\Lambda_F(S_{\star})}$ by
the first part of the paper. It also follows that the functor in
the other direction from $\mathcal{M}_{et}(\Lambda(P_+))$ to
$\mathcal{M}_{o-tor}(P)$ sends $D^0(V_{St})$ back to $V_{St}$. We
expect that the method of computing $D^i(V_{St})$ for $i\geq 0$
generalizes to a wider class of smooth $o$-torsion
$P$-representations. For technical reasons we restrict ourselves in this section to the general linear group $\GL_{d+1}(\mathbb{Q}_p)$ with $d\geq 1$. The case $\GL_2(\mathbb{Q}_p)$ is also formally included---however, the functor $D$ is known \cite{Co1,Co2} to be exact in this case.

\section{Preliminaries and notations}

\subsection{Basic notations}

We are going to use the notations of \cite{SV}, but for the convenience of the reader we recall them here, as well. Let $G$ be the group of $\mathbb{Q}_p$-rational points of a $\mathbb{Q}_p$-split connected reductive group over $\mathbb{Q}_p$. Assume further that the centre of this reductive group is also connected. We fix a Borel subgroup $P=TN$ in $G$ with maximal split torus $T$ and unipotent radical $N$. Let $\Phi^+$ denote, as usual, the set of roots of $T$ positive with respect to $P$ and let $\Delta\subseteq\Phi^+$ be the subset of simple roots. For any $\alpha\in\Phi^+$ we have the root subgroup $N_{\alpha}\subseteq N$. We recall that $N=\prod_{\alpha\in\Phi^+}N_{\alpha}$ for any total ordering of $\Phi^+$. Let $T_0\subseteq T$ be the maximal compact subgroup. We fix a compact open subgroup $N_0\subseteq N$ which is totally decomposed, ie.\ $N_0=\prod_{\alpha}N_0\cap N_{\alpha}$ for any total ordering of $\Phi^+$. Then $P_0:=T_0N_0$ is a group. We introduce the submonoid $T_+\subseteq T$ of all $t\in T$ such that $tN_0t^{-1}\subseteq N_0$, or equivalently, such that $\alpha(t)$ is integral for any $\alpha\in\Delta$. Then $P_+:=N_0T_+=P_0T_+P_0$ is obviously a submonoid of $P$.

We also fix a finite extension $K/\mathbb{Q}_p$ with ring of integers $o$, prime element $\pi$, and residue class field $k$. For any profinite group $H$ let $\Lambda(H):=o\bs H\js$, resp.\ $\Omega(H):=k\bs H\js=\Lambda(H)/\pi\Lambda(H)$ be the Iwasawa algebra of $H$ with coefficients in $o$, resp.\ $k$.

\subsection{The functors $D$ and $D^i$}

By a representation we will always mean a linear action of the group (or monoid) in question in a torsion $o$-module $V$. It is called smooth if the stabilizer of each element in $V$ is open in the group. We put $V^*:=\Hom_o(V,K/o)$ the Pontryagin dual of $V$ which is a compact linear-topological $o$-module. Following \cite{SV} we define
\begin{equation*}
D(V):=\varinjlim_{M} M^*
\end{equation*}
where $M$ runs through all the generating $P_+$-subrepresentations of $V$. Whenever $V$ is compactly induced it is equipped with an action of the ring $\Lambda(P_+)$ which is by definition the image of the natural map
\begin{equation*}
\Lambda(P_0)\otimes_{o[P_0]}o[P_+]\rightarrow \varprojlim_{Q}o[Q\setminus P_+]
\end{equation*}
where $Q$ runs through all open normal subgroups $Q\subseteq P_0$ which satisfy $bQb^{-1}\subseteq Q$ for any $b\in P_+$ (cf.\ Proposition 3.4 in \cite{SV}). The $\Lambda(P_+)$-modules $D^i(V)$ for a general smooth $P$-representation $V$ and $i\geq 0$ are obtained as the cohomology groups $D^i(V):=h^i(D(\mathcal{I}_{\bullet}(V)))$ for some resolution
\begin{equation*}
\mathcal{I}_{\bullet}(V)\colon\dots\rightarrow \ind_{P_0}^{P}(V_n)\rightarrow\dots\rightarrow\ind_{P_0}^{P}(V_0)\rightarrow V\rightarrow0
\end{equation*}
of $V$ by compactly induced representations. This is independent of the choice
of the resolution by Corollary 4.4 in \cite{SV}. Note that $D(V)\subsetneq D^0(V)$ in general.

\subsection{The ring $\Lambda_{\ell}(N_0)$}

As in \cite{SV} we fix once and for all isomorphisms of algebraic groups
\begin{equation*}
\iota_{\alpha}\colon N_{\alpha}\overset{\cong}{\rightarrow}\mathbb{Q}_p
\end{equation*}
for $\alpha\in\Delta$, such that
\begin{equation*}
\iota_{\alpha}(tnt^{-1})=\alpha(t)\iota_{\alpha}(n)
\end{equation*}
for any $n\in N_{\alpha}$ and $t\in T$. Since $\prod_{\alpha\in\Delta}N_{\alpha}$ is naturally a quotient of $N/[N,N]$ we now introduce the group homomorphism
\begin{equation*}
\ell:=\sum_{\alpha\in\Delta}\iota_{\alpha}\colon N\rightarrow\mathbb{Q}_p.
\end{equation*}
Moreover, for the sake of convenience we normalize the $\iota_{\alpha}$ such that
\begin{equation*}
\iota_{\alpha}(N_0\cap N_{\alpha})=\mathbb{Z}_p
\end{equation*}
for any $\alpha$ in $\Delta$. In particular, we then have $\ell(N_0)=\mathbb{Z}_p$. We put $N_1:=\Ker(\ell_{\mid N_0})$. The group homomorphism $\ell$ also induces a map
\begin{equation*}
\Lambda(N_0)\twoheadrightarrow\Lambda(\mathbb{Z}_p)
\end{equation*}
which we still denote by $\ell$. By \cite{CFKSV} the multiplicatively closed subset $S:=\Lambda(N_0)\setminus(\pi,\Ker(\ell))$ is a left and right Ore set in $\Lambda(N_0)$ and we may define the localization $\Lambda(N_0)_S$ of $\Lambda(N_0)$ at $S$. We define the ring $\Lambda_{\ell}(N_0):=\Lambda_{N_1}(N_0)$ as the completion of $\Lambda(N_0)_S$ with respect to the ideal $(\pi,\Ker(\ell))\Lambda(N_0)_S$. This is a strict-local ring with maximal ideal $(\pi,\Ker(\ell))\Lambda_{\ell}(N_0)$. Moreover, it is pseudocompact (c.f.\ Thm 4.7 in \cite{SVe}).

\subsection{Generalized $(\varphi,\Gamma)$-modules}
Now since we assume that the centre of $G$ is connected, the quotient $X^*(T)/\bigoplus_{\alpha\in\Delta}\mathbb{Z}\alpha$ is free. Hence we find a cocharacter $\xi$ in $X_*(T)$ such that $\alpha\circ\xi=\id_{\mathrm{G}_m}$ for any $\alpha$ in $\Delta$. It is injective and uniquely determined up to a central cocharacter. We fix once and for all such a $\xi$. It satisfies
\begin{equation*}
\xi(\mathbb{Z}_p\setminus\{0\})\subseteq T_+
\end{equation*}
and
\begin{equation*}
\ell(\xi(a)n\xi(a^{-1}))=a\ell(n)
\end{equation*}
for any $a$ in $\mathbb{Q}_p^{\times}$ and $n$ in $N$. Put $\Gamma:=\xi(1+p^{\epsilon(p)}\mathbb{Z}_p)$ and $\varphi:=\xi(p)$.

The group $\Gamma$ and the semigroup generated by $\varphi$ naturally act on the ring $\Lambda_{\ell}(N_0)$. Hence we may define $(\varphi,\Gamma)$-modules (resp.\ $\varphi$-modules) over $\Lambda_{\ell}(N_0)$ as $\Lambda_{\ell}(N_0)$-modules together with a commuting and compatible action of $\varphi$ and $\Gamma$ (resp.\ just a compatible action of $\varphi$). The notion of $(\Lambda_{\ell}(N_0),\Gamma,\varphi)$-module refers to $(\varphi,\Gamma)$-modules that are \emph{finitely generated} over $\Lambda_{\ell}(N_0)$.
We call a $\varphi$-module $M$ \'etale if the map
\begin{eqnarray*}
\Lambda_{\ell}(N_0)\otimes_{\varphi}M&\rightarrow& M\\
\nu\otimes m&\mapsto&\nu\varphi_M(m)
\end{eqnarray*}
is bijective.

The map $\ell$ induces a $\varphi$- and $\Gamma$-equivariant ring homomorphism
\begin{equation*}
\Lambda_{\ell}(N_0)\twoheadrightarrow\Lambda_F(\mathbb{Z}_p)
\end{equation*}
onto Fontaine's ring $\Lambda_F(\mathbb{Z}_p)$ which is the $p$-adic completion of the Laurent series ring $o\bs T\js[T^{-1}]$. Hence it gives rise to a functor from (\'etale) $(\varphi,\Gamma)$-modules over $\Lambda_{\ell}(N_0)$ to not necessarily finitely generated (\'etale) $(\varphi,\Gamma)$-modules over $\Lambda_F(\mathbb{Z}_p)$. We may restrict this functor to pseudocompact (or less generally to finitely generated) \'etale modules. The main result of this short note is that this restriction is exact.

\section{Exactness of reduction on pseudocompact modules}

\subsection{A $p$-valuation on $N_0$}

In this section we are going to define a $p$-valuation $\omega$ on the group
$N_0$ which we will restrict to $N_1$ in order to produce ideals $J_n$ in $\Lambda_{\ell}(N_0)$. Since $N_0$ is
totally decomposed we can fix topological generators
$n_{\alpha}$ of $N_0\cap N_{\alpha}$ for any $\alpha$ in $\Delta$
such that $\ell(n_{\alpha})=1$. Further, we fix topological
generators $n_{\beta}$ of $N_0\cap N_{\beta}$ for each
$\beta\in\Phi^+\setminus\Delta$. Now we define a $p$-valuation $\omega$ on $N_0$ as follows. Any $\beta$ in
$\Phi^+$ can be written as a positive integer combination
$\beta=\sum_{\alpha\in\Delta}m_{\alpha\beta}\alpha$ of simple roots
$\alpha$. We denote by $m_{\beta}:=\sum_{\alpha}m_{\alpha\beta}$ the degree of
$\beta\circ\xi$ which is a positive integer and is equal to $1$ if and only if
$\beta$ lies in $\Delta$. Further, we fix a total ordering $<$ of $\Phi^+$
such that whenever $m_{\alpha}<m_{\beta}$ for roots $\alpha,\beta$ in $\Phi^+$ then also
$\alpha<\beta$. As $N_0$ is totally decomposed we may write any element $g$ in
$N_0$ uniquely as a product
\begin{equation*}
g=\prod_{\alpha\in\Phi^+}n_{\alpha}^{g_{\alpha}}
\end{equation*}
where $g_{\alpha}$ are in $\mathbb{Z}_p$ and
the product is taken in the ordering $<$ of $\Phi^+$ defined above. We put 
\begin{equation*}
\omega(g):=\min_{\alpha\in\Phi^+}m_{\alpha}(v_p(g_{\alpha})+1)
\end{equation*}
for any $1\neq g$ in $N_0$, and $\omega(1):=\infty$. Here $v_p$ denotes the
additive $p$-adic valuation on $\mathbb{Z}_p$. 

\begin{lem}\label{11}
The function $\omega$ is a $p$-valuation on $N_0$. In other words, we have
\begin{enumerate}[$(i)$]
\item $\omega(gh^{-1})\geq \min(\omega(g),\omega(h))$.
\item $\omega(g^{-1}h^{-1}gh)\geq \omega(g)+\omega(h)$.
\item $\omega(g^p)=\omega(g)+1$.
\end{enumerate}
\end{lem}
\begin{proof}
For the proof of $(i)$ we are going to use triple induction. At first by
induction on the number of non-zero coordinates among
$(h_{\alpha})_{\alpha\in\Phi}$ we are reduced to the case when $h$ is of the
form $n_{\alpha}^{h_{\alpha}}$ for some $\alpha$ in $\Phi^+$. Let now $g$ be
of the form $g=\prod_{k=1}^rn_{\beta_k}^{g_{\beta_k}}$ with
$\beta_1<\beta_2<\cdots<\beta_r$ in $\Phi^+$. If
$\beta_r\leq \alpha$ then $(i)$ is trivially satisfied. On the other hand, if
$\beta_r>\alpha$ then we write 
\begin{equation*}
\prod_{k=1}^rn_{\beta_k}^{g_{\beta_k}}\cdot n_{\alpha}^{-h_{\alpha}}=\prod_{k=1}^{r-1}n_{\beta_k}^{g_{\beta_k}}n_{\alpha}^{-h_{\alpha}}\cdot
n_{\beta_r}^{g_{\beta_r}}[n_{\beta_r}^{g_{\beta_r}},n_{\alpha}^{-h_{\alpha}}].
\end{equation*} 
Now we use
(descending) induction on $\alpha$ in the chosen ordering of $\Phi^+$ and suppose 
that the statement $(i)$ is true for any $g$ and any $h'$ of the form
$h'=n_{\alpha'}^{h'_{\alpha'}}$ with $\alpha'>\alpha$. For this we remark that the
set $\Phi^+$ is finite and totally ordered. Note that for
any $\alpha<\beta$ in $\Phi^+$ we have 
\begin{equation}\label{12}
[n_{\beta}^{g_{\beta}},n_{\alpha}^{-h_{\alpha}}]=\prod_{i\beta+j\alpha\in\Phi^+;
  i,j>0}n_{i\beta+j\alpha}^{c_{\beta,\alpha,i,j}g_{\beta}^{i}(-h_{\alpha})^j}
\end{equation}
by the commutator formula in Proposition 8.2.3 in \cite{Sp}. Here the
constants $c_{\beta,\alpha,i,j}$ a priori only lie in $\mathbb{Q}_p$ but since
our parametrization and the fact that $N_0$ is a subgroup of $N$ they are
forced to be in $\mathbb{Z}_p$. Moreover, we have $m_{i\beta+j\alpha}\geq
m_{\beta}+m_{\alpha}> m_{\alpha}$, so we may apply the inductional hypothesis
for $n_{\beta_r}^{g_{\beta_r}}$ and for all the terms on the right hand side
of \eqref{12} (with the choice $\beta:=\beta_r$) in order to obtain
\begin{equation}\label{14}
\omega(gh^{-1})\geq\min\left(\omega\left(\prod_{k=1}^{r-1}n_{\beta_k}^{g_{\beta_k}}n_{\alpha}^{-h_{\alpha}}\right),
\omega(n_{\beta_r}^{g_{\beta_r}}),\omega\left([n_{\beta_r}^{g_{\beta_r}},n_{\alpha}^{-h_{\alpha}}]\right)\right).
\end{equation}
On the other hand, we compute
\begin{eqnarray}
\omega(n_{i\beta+j\alpha}^{c_{\beta,\alpha,i,j}g_{\beta}^{i}(-h_{\alpha})^j})=m_{i\beta+j\alpha}(v_p(c_{\beta,\alpha,i,j}g_{\beta}^{i}
(-h_{\alpha})^j)+1)\geq\label{13}\\
\geq (m_{\beta}+m_{\alpha})(v_p(g_{\beta})+v_p(h_{\alpha})+1)\geq\omega(n_{\beta}^{g_{\beta}})+\omega(n_{\alpha}^{h_{\alpha}})\geq
\min(\omega(n_{\beta}^{g_{\beta}}),\omega(n_{\alpha}^{h_{\alpha}})).\notag
\end{eqnarray}
Hence combining \eqref{14} and \eqref{13} (with the choice $\beta:=\beta_r$) we
are done by a third induction on $r$. 

Since we know $(i)$ it suffices to check $(ii)$ in the case
$g=n_{\beta}^{g_{\beta}}$ and
$h=n_{\alpha}^{h_{\alpha}}$. For these we are done by \eqref{12} and
\eqref{13}. The assertion $(iii)$ is clear from the definition using $(i)$ and
$(ii)$. 
\end{proof}

\subsection{The ideals $J_n$}

In view of Lemma \ref{11} we define for each positive integer $n$ the normal
subgroup $N_{0,n}$ in $N_0$ as the set of elements
$g$ in $N_0$ with $\omega(g)\geq n$. Further we put $N_{1,n}:=N_1\cap
N_{0,n}$. In particular, $N_{1,1}=N_1$ and $N_{0,1}=N_0$. We define
$J_n(\Lambda(N_1))$ to be the kernel of the natural surjection
$\Lambda(N_1)\twoheadrightarrow\Lambda(N_1/N_{1,n})$. Moreover, we denote by
$J_n$ the ideal generated by $J_n(\Lambda(N_1))$ in $\Lambda_{\ell}(N_0)$. We
further have the following 

\begin{lem}\label{1}
$N_{1,n}$ is a normal subgroup in $P_0$ for any $n\geq 1$. In particular, $J_n$ is the kernel of the natural surjection from $\Lambda_{\ell}(N_0)=\Lambda_{N_1}(N_0)$ onto $\Lambda_{N_1/N_{1,n}}(N_0/N_{1,n})$. Further, we have $\varphi N_{1,n}\varphi^{-1}\subseteq N_{1,n+1}$. Therefore there is an induced $\varphi$-action on each $\Lambda_{\ell}(N_0)/J_n$ such that the module $J_n/J_{n+1}$ is killed by $\varphi$ for any $n\geq 1$.
\end{lem}
\begin{proof}
Since $N_1$ is normal in $P_0$ it suffices to verify that $N_{0,n}$ is normal
in $P_0$. For this note that $N_{0,n}$ is normal in $N_0$ as $\omega$ is a
$p$-valuation and if $t$ is in $T_0$ then we have
  $tn_{\alpha}t^{-1}=n_{\alpha}^{t_{\alpha}}$ with $t_{\alpha}$ in
  $\mathbb{Z}_p^{\times}$. Hence the first part of the statement. For the
  second part we note that $\varphi n_{\alpha}\varphi^{-1}=n_{\alpha}^{p^{m_\alpha}}$. 
\end{proof}

Note that the Jacobson radical $Jac(\Lambda_{\ell}(N_0))$ is equal to the
ideal $(\pi,J_1)$ by definition of $J_1$. Moreover, any element $g$ in
$N_{1,n\max_{\beta\in\Phi^+}m_{\beta}}$ is a product of $p^n$th powers of
elements in $N_1$, hence $J_{n\max_{\beta\in\Phi^+}m_{\beta}}\subseteq
Jac(\Lambda_{\ell}(N_0))^{n+1}$. Indeed, if $g$ is any element in 
$N_1$ then $g^{p^n}-1=\prod_{j=0}^n\Phi_{p^j}(g)$ (with $\Phi_{p^j}(x)$ being
the $p^j$th cyclotomic polynomial) and $\Phi_{p^j}(g)$ clearly lies in
$Jac(\Lambda_{\ell}(N_0))$ as both $p$ and $g-1$ lie in $Jac(\Lambda_{\ell}(N_0))$. In particular, $\bigcap_n J_n=0$.

Recall that $\Lambda_{\ell}(N_0)$ is a pseudocompact ring (c.f.\
\cite{SVe} Thm. 4.7).

\begin{lem}\label{2}
If $M$ is any pseudocompact module over $\Lambda_{\ell}(N_0)$ then
$J_nM$ and $M/J_nM$ are also pseudocompact in the subspace, resp.
quotient topologies.
\end{lem}
\begin{proof}
It suffices to show that $J_nM$ is closed in $M$. By Lemma 1.6 in
\cite{B} and by the fact that the pseudocompact modules form an abelian category (\cite{G} IV.3.\ Thm.\ 3) we are reduced to the case when $M=\prod_{i\in
I}\Lambda_{\ell}(N_0)$ with the product topology. However, in this
case we have $J_n\prod_{i\in I}\Lambda_{\ell}(N_0)=\prod_{i\in
I}J_n$ as $J_n$ is finitely generated ($\Lambda_{\ell}(N_0)$ is
noetherian), and this is closed in the product topology as $J_n$
is closed in $\Lambda_{\ell}(N_0)$ using once again that it is finitely
generated and hence pseudocompact in the subspace topology of $\Lambda_{\ell}(N_0)$.
\end{proof}

\begin{lem}\label{4}
If $M$ is any pseudocompact module over the ring
$\Lambda_{\ell}(N_0)$ then the natural map induces an isomorphism
\begin{equation*}
M\cong\varprojlim_n M/J_nM.
\end{equation*}
\end{lem}
\begin{proof}
By Lemma \ref{2} the submodules $J_nM$ are closed, and since
$J_{n\max_{\beta\in\Phi^+}m_{\beta}}\subseteq
Jac(\Lambda_{\ell}(N_0))^{n}$ we have $\bigcap_n J_nM=0$. The
statement follows from IV.3.\ Proposition 10 in \cite{G}.
\end{proof}

\subsection{Main result}

\begin{pro}\label{10}
Let $M$ and $N$ be pseudocompact \'etale $\varphi$-modules over
$\Lambda_{\ell}(N_0)$. Then injective continuous maps (in the pseudocompact topology)
$M\hookrightarrow N$ reduce to injective maps
$M/J_1M\hookrightarrow N/J_1N$ between the $\varphi$-modules over
$\Lambda_F(\mathbb{Z}_p)$.
\end{pro}
\begin{proof}
Let $K_n$ be the kernel of the induced map from $M/J_nM$ to
$N/J_nN$. We assume indirectly that $K_1\neq 0$. We show that the
natural map from $K_{n}$ to $K_{1}$ is surjective for any $n$. For this we are going to use the following commutative diagram with some $X_n$ and $Y_n$.

\begin{equation}\label{3}
\begin{CD}
& & 0 & & 0 & & 0 \\
& & @VVV @VVV @VVV\\
0 @>>> X_n @>>> K_{n} @>>> K_1\\
& & @VVV @VVV @VVV\\
0 @>>> J_1M/J_{n}M @>>> M/J_{n}M @>>> M/J_1M @>>> 0\\
& & @VVV @VVV @VVV\\
0 @>>> J_1N/J_{n}N @>>> N/J_{n}N @>>> N/J_1N @>>> 0\\
& & @VVV\\
& & Y_n\\
& & @VVV\\
& & 0
\end{CD}
\end{equation}

We remark immediately that by Lemma \ref{2} all the modules in the
diagram \eqref{3} are pseudocompact modules over
$\Lambda_{\ell}(N_0)$, and all the maps are continuous in the
pseudocompact topologies. Indeed, the pseudocompact modules form
an abelian category (\cite{G} IV.3.\ Thm.\ 3).

By the snake lemma we obtain the exact sequence

\begin{equation}\notag
0\rightarrow X_n\rightarrow K_{n}\rightarrow
K_1\overset{\delta_n}{\rightarrow} Y_n.
\end{equation}

We claim that there does not exist any nonzero $\varphi$-equivariant $\Lambda_{\ell}(N_0)$-homomorphism from $K_1$ to
$Y_n$. This would show that $K_{n}$ surjects onto $K_1$ for any
$n$. Note that $\Lambda_{\ell}(N_0)$ acts on $K_1$ via its quotient
$\Lambda_F(\mathbb{Z}_p)$ hence we may view $K_1$ as a $\varphi$-module over both rings.
As $\varphi$ is flat over $\Lambda_F(\mathbb{Z}_p)$, \'etale
$\varphi$-modules form an abelian category over $\Lambda_F(\mathbb{Z}_p)$. In 
particular, $K_1$ is \'etale as a $\varphi$-module over
$\Lambda_F(\mathbb{Z}_p)$ since it is the kernel of a homomorphism
between the \'etale modules $M/J_1M$ and $N/J_1N$. We remark that $K_1$ is
not \'etale as a $\varphi$-module over $\Lambda_{\ell}(N_0)$ since $K_1$ is
annihilated by $J_1$ and
$\Lambda_{\ell}(N_0)\otimes_{\varphi,\Lambda_{\ell}(N_0)}K_1$ is not. The
latter is only annihilated by $\varphi(J_1)\subseteq J_2\subsetneq J_1$. So
the map
\begin{equation*}
\Lambda_{\ell}(N_0)\otimes_{\varphi,\Lambda_{\ell}(N_0)}K_1\rightarrow K_1
\end{equation*}
is surjective (since $K_1$ is \'etale over $\Lambda_F(\mathbb{Z}_p)$), but not injective. Therefore if
there is a surjective $\varphi$-equivariant
$\Lambda_{\ell}(N_0)$-homomorphism from $K_1$ to some $\varphi$-module $A$
over $\Lambda_{\ell}(N_0)$
then we also have that $\varphi(A)$ generates $A$ as a
$\Lambda_{\ell}(N_0)$-module. On the other hand, $J_1N/J_nN$
admits the filtration $Fil^k(J_1N/J_nN):=J_kN/J_nN$ for $1\leq
k\leq n$. This induces a filtration $Fil^k(Y_n)$ on $Y_n$ via the above
surjection in \eqref{3}. Let us assume now that $\delta_n$ is nonzero. Then
there is an integer $k<n$ such that $\delta_n(K_1)\subseteq
Fil^k(Y_n)$ but $\delta_n(K_1)\not\subseteq Fil^{k+1}(Y_n)$. Hence
we get a nonzero $\varphi$-equivariant $\Lambda_{\ell}(N_0)$-homomorphism from $K_1$ to $Fil^k(Y_n)/Fil^{k+1}(Y_n)$
which we denote by $\delta_n'$. However, we claim that $\varphi$ acts as zero on the latter which
will contradict to the fact that $\varphi(\delta_n'(K_1))$
generates $\delta_n'(K_1)$. Indeed, we have a surjective
composite map
\begin{equation}\notag
(J_k/J_{k+1})\otimes_{\Lambda_{\ell}(N_0)}N\twoheadrightarrow
J_kN/J_{k+1}N\twoheadrightarrow Fil^k(Y_n)/Fil^{k+1}(Y_n),
\end{equation}
hence $\varphi(Fil^k(Y_n)/Fil^{k+1}(Y_n))=0$ as we have
$\varphi(J_k)\subseteq J_{k+1}$ by Lemma \ref{1}.

Now we have a map from the projective system $(K_n)_n$ to the
projective system $(K_1)_n$ which is surjective on each layer,
hence its projective limit is also surjective by the exactness of
$\varprojlim$ on pseudocompact modules (\cite{G} IV.3.\ Thm.\ 3).
The statement follows from the completeness of $M$ (Lemma
\ref{4}).
\end{proof}

Whenever $M$ and $N$ are finitely generated over
$\Lambda_{\ell}(N_0)$ then they admit a unique pseudocompact
topology (since $\Lambda_{\ell}(N_0)$ is pseudocompact and
noetherian \cite{SVe} Thm. 4.7 and Lemma 4.2(ii)) and any
homomorphism between them is continuous. So we obtain the main result of this paper as corollary of Proposition \ref{10}.
\begin{pro}
The functor from the category of \'etale $(\Lambda_{\ell}(N_0),\Gamma,\varphi)$-modules to the category of \'etale $(\Lambda_{F}(S_0),\Gamma,\varphi)$-modules induced by the natural surjection
\begin{equation*}
\ell\colon\Lambda_{\ell}(N_0)\twoheadrightarrow \Lambda_{F}(\mathbb{Z}_p)
\end{equation*}
is exact.
\end{pro}

\subsection{Vanishing of higher $\Tor$-functors}\label{30}

Let $M$ be a pseudocompact \'etale $\varphi$-module over
$\Lambda_{\ell}(N_0)$. Then $M/(\pi,J_1)M$ is also a pseudocompact
\'etale $\varphi$-module over the field
$\Lambda_{\ell}(N_0)/(\pi,J_1)\cong k\bg t\jg$. Hence
there is an index set $I$ such that we have an isomorphism of pseudocompact modules
\begin{equation*}
M/(\pi,J_1)M\cong \prod_{i\in I}\Lambda_{\ell}(N_0)/(\pi,J_1)
\end{equation*}
by Lefschetz's Structure Theorem for
linearly compact vector spaces (\cite{L}, p.\ 83 Thm.\ (32.1), see also
\cite{D}). Note that over fields the notion of linearly compact vectorspaces
coincides with the notion of pseudocompact modules. Indeed, pseudocompact
modules over fields are by definition the projective limits of finite dimensional
vectorspaces with the projective limit topology of the discrete topology on
each finite dimensional vectrospace. However, the category of linearly compact
vectorspaces is closed under products and factors by closed subspaces
(properties $c)$ and $d)$ on page 1 of \cite{D}), hence
also under projective limits.
Moreover, we have $(\pi,J_1)=Jac(\Lambda_{\ell}(N_0))$, therefore we
obtain a projective cover of $M$
\begin{equation}
f\colon\prod_{i\in I}\Lambda_{\ell}(N_0)\twoheadrightarrow
M\label{7}
\end{equation}
which is an isomorphism modulo $(\pi,J_1)$. Note that in fact $f$ is a right minimal
morphism. Indeed, whenever $g\colon \prod_{i\in I}\Lambda_{\ell}(N_0)\rightarrow
\prod_{i\in I}\Lambda_{\ell}(N_0)$ is a $\Lambda_{\ell}(N_0)$-homomorphism
such that $f\circ g=f$ then $g$ is the identity modulo
$Jac(\Lambda_{\ell}(N_0))=(\pi,J_1)$. Hence $g$ is invertible by the
completeness of $\prod_{i\in I}\Lambda_{\ell}(N_0)$ with respect to the
$Jac(\Lambda_{\ell}(N_0))$-adic filtration.

In this section we need to assume that $\varphi$ acts continuously
on the pseudocompact module $M$. Note that this is automatic if
$M$ is finitely generated over $\Lambda_{\ell}(N_0)$.

\begin{lem}\label{8}
Let $F=\prod_{i\in I}\Lambda_{\ell}(N_0)$ be a $\varphi$-module
over $\Lambda_{\ell}(N_0)$ with continuous $\varphi$-action. Then $F$ is \'etale if and only if so
is $F/Jac(\Lambda_{\ell}(N_0))F$ over $k\bg t\jg$.
\end{lem}
\begin{proof}
If $F$ is \'etale then by definition so is
$F/Jac(\Lambda_{\ell}(N_0))F$. Now assume that
$F/Jac(\Lambda_{\ell}(N_0))F$ is \'etale. In other words the map
\begin{equation}
1\otimes\varphi\colon\Lambda_{\ell}(N_0)\otimes_{\varphi,\Lambda_{\ell}(N_0)}F\rightarrow
F\label{6}
\end{equation}
is isomorphism modulo $Jac(\Lambda_{\ell}(N_0))$. Therefore
\eqref{6} is for instance surjective as its cokernel is
pseudocompact and killed by $Jac(\Lambda_{\ell}(N_0))$. On the other hand, since $F$ is topologically free, we have a continuous section of the map \eqref{6}. Since \eqref{6} is an isomorphism modulo $Jac(\Lambda_{\ell}(N_0))$, so is this section. However, by the same argument as above this section also has to be surjective and therefore is an inverse to the map \eqref{6}.
\end{proof}

\begin{pro}\label{9}
Let $M$ be an \'etale pseudocompact $\varphi$-module over
$\Lambda_{\ell}(N_0)$ with continuous $\varphi$-action. Then the
action of $\varphi$ on $M$ can be lifted continuously to $F:=\prod_{i\in
I}\Lambda_{\ell}(N_0)$ via the surjection $f$ in \eqref{7}. Any
such lift makes $F$ an \'etale $\varphi$-module.
\end{pro}
\begin{proof}
Let us define another continuous
$\Lambda_{\ell}(N_0)$-homomorphism
\begin{eqnarray*}
g\colon\prod_{i\in I}\Lambda_{\ell}(N_0)&\rightarrow& M\\
e_i&\mapsto&\varphi(f(e_i)).
\end{eqnarray*}
We need to check that $\lim_{i\in I}\varphi(f(e_i))=0$ in the
pseudocompact topology of $M$ so that $g$ really defines a
continuous homomorphism. This is, however, clear by the continuity
of $\varphi$ and $f$. By the projectivity of $F$ (Lemma 1.6 in
\cite{B}) we obtain a lift $\varphi_{\mathrm{lin}}$
\begin{equation*}
\begin{xy}
  \xymatrix{
      & F \ar[d]^f \\
      F \ar@{.>}[ru]^{\varphi_{\mathrm{lin}}} \ar[r]^g & M
      }
\end{xy}
\end{equation*}
which we define as the linearization of $\varphi$ on $F$. Hence we
define
\begin{equation*}
\varphi(e_i):=\varphi_{\mathrm{lin}}(e_i)
\end{equation*}
and extend it $\sigma_{\varphi}$-linearly and continuously to the
whole $F$. By construction this is a lift of $\varphi_{|M}$. The
\'etaleness follows from Lemma \ref{8} noting that by construction
of \eqref{7} we have
$F/Jac(\Lambda_{\ell}(N_0))F=M/Jac(\Lambda_{\ell}(N_0))M$ and the
latter is \'etale as so is $M$.
\end{proof}

\begin{cor}
For any pseudocompact \'etale $\varphi$-module $M$ over
$\Lambda_{\ell}(N_0)$ with continuous $\varphi$ and any $i\geq1$
we have
\begin{equation*}
\Tor^i_{\Lambda_{\ell}(N_0)}(\Lambda_{\ell}(N_0)/J_1,M)=0.
\end{equation*}
\end{cor}
\begin{proof}
By Proposition \ref{9} there is a projective resolution
$(F_i)_{i\in\mathbb{N}}$ of $M$ in the category of pseudocompact
$\Lambda_{\ell}(N_0)$-modules, such that the $F_i$ are \'etale
$\varphi$-modules and the resolution is $\varphi$-equivariant. By
Proposition \ref{10}, the functor
$\Lambda_{\ell}(N_0)/J_1\otimes_{\Lambda_{\ell}(N_0)}\cdot$ is
exact on this resolution. The result follows noting that the
modules $\prod_{i\in I}\Lambda_{\ell}(N_0)$ are flat over
the noetherian ring $\Lambda_{\ell}(N_0)$ as in this case an arbitrary direct product of flat modules is flat again.
\end{proof}

\begin{cor}
Let $M$ be a pseudocompact \'etale module over
$\Lambda_{\ell}(N_0)$ with continuous $\varphi$ such that $\pi
M=0$. Then there exists an index set $I$ such that $M\cong
\prod_{i\in I}\Lambda_{\ell}(N_0)/\pi$. In particular, $M$ is a
projective object in the category of pseudocompact modules over
$\Lambda_{\ell}(N_0)/\pi$.
\end{cor}
\begin{proof}
By Proposition \ref{9} we obtain a minimal projective cover $F$ of
$M$ with $F$ admitting an \'etale lift of the $\varphi$-action on
$M$. Since $\pi M=0$ this factors through $F/\pi F$ which is also
\'etale in the induced $\varphi$-action. Now we denote by $K$ the
kernel of the map from $F/\pi F$ onto $M$. Then $K$ is also \'etale
as these form an abelian category. Hence by Proposition \ref{10}
we obtain an exact sequence
\begin{equation*}
0\rightarrow K/J_1K\rightarrow F/(\pi,J_1)F\rightarrow
M/J_1M\rightarrow 0.
\end{equation*}
However, the map $F/(\pi,J_1)F\rightarrow M/J_1M$ is an isomorphism
by the construction of $F$ \eqref{7} showing that $K/J_1K=0$
whence $K=0$ as $K$ is pseudocompact.
\end{proof}

\section{An example}\label{31}

In this section we are going to investigate the so called Steinberg representation. For the sake of simplicity (of the Bruhat-Tits building) we let $G$ be $\GL_{d+1}(\mathbb{Q}_p)$ in this section for some $d\geq 1$ and $P$ be its standard Borel subgroup of lower triangular matrices. Recall that the group $P=NT$ acts on $N$ by $(nt)(n')=ntn't^{-1}$. This induces an action of $P$ on the vector space $V_{St}:=C^{\infty}_c(N)$ of $k$-valued locally constant functions with compact support on $N$. It is straightforward to see (cf.\ Example on p.\ 8 in \cite{SV} and \cite{Vi} Lemme 4) that the subspace $M:=C^{\infty}(N_0)$ of locally constant functions on $N_0$ is generating and $P_+$-invariant. Moreover, it is shown in \cite{SV} Lemma 2.6 that we have $D(V_{St})=\Lambda(N_0)/\pi\Lambda(N_0)$. We have the following refinement of this.

\begin{pro}\label{e16}
Let $V_{St}$ be the smooth modulo $p$ Steinberg representation of the group $P$. Then we have $D^0(V_{St})=D(V_{St})=\Lambda(N_0)/\pi\Lambda(N_0)$, and $D^i(V_{St})=0$ for any $i\geq 1$.
\end{pro}

For the proof of Proposition \ref{e16} we are going to construct an explicit resolution
\begin{equation*}
\mathcal{I}_{\bullet}\colon0\rightarrow\ind_{P_0Z}^P(V_d)\rightarrow\dots\rightarrow\ind_{P_0Z}^P(V_1)\rightarrow\ind_{P_0Z}^P(V_0)\rightarrow V_{St}\rightarrow0
\end{equation*}
of $V_{St}$ using the Bruhat-Tits building of $G$. Here $Z$ denotes the centre of $G$ that will act trivially on each $V_i$ ($0\leq i\leq d$). Since $Z\cong \mathbb{Q}_p^{\times}$, Lemma 11.8 in \cite{SV} generalizes to this case with the same proof, so we have $D^0(\ind_{P_0Z}^P(V_i))=D(\ind_{P_0Z}^P(V_i))$ and $D^i(\ind_{P_0Z}^P(V_i))=0$ for all $0\leq i\leq d$. In particular, we may compute $D^i(V_{St})=h^{i}(D(\mathcal{I}_{\bullet}))$.

Recall that the Bruhat-Tits building $\mathcal{BT}$ of $G$ is the simplicial complex whose vertices are the similarity classes $[L]$ of $\mathbb{Z}_p$-lattices in the vector space $\mathbb{Q}_p^{d+1}$ and whose $q$-simplices are given by families $\{[L_0],\dots,[L_q]\}$ of similarity classes such that
\begin{equation*}
pL_0\subsetneq L_1\subsetneq\dots\subsetneq L_q\subsetneq L_0.
\end{equation*}
Let $\mathcal{BT}_q$ denote the set of all $q$-simplices of $\mathcal{BT}$. We also fix an orientation of $\mathcal{BT}$ with the corresponding incidence numbers $[\eta:\eta']$. We choose a basis $e_0,\dots,e_d$ of $\mathbb{Q}_p^{d+1}$ in which $P$ is the Borel subgroup of lower triangular matrices and denote the origin of $\mathcal{BT}$ by $x_0:=[\sum_{i=0}^d\mathbb{Z}_pe_i]$. Further, for all $1\leq i\leq d$ let $\varphi_i$ be the dominant diagonal matrix $\mathrm{diag}(1,\dots,1,p,\dots,p)$ with $i$ entries equal to $1$ and $d+1-i$ entries equal to $p$ and put $x_i:=\varphi_ix_0$. Then $T_+/T_0Z$ is clearly generated by the elements $\{\varphi_iT_0Z\}_{i=1}^d$ as a monoid. Moreover, for each subset $J=\{j_1<\dots<j_q\}\subseteq\{1,\dots,d\}$ we define the (oriented) $q$-simplex
\begin{equation*}
\eta_J:=\{x_0,x_{j_1},\dots,x_{j_q}\}.
\end{equation*}
Now we define the coefficient system
\begin{equation*}
V_{nt\eta_J}:=C_c^{\infty}\left(nt\left(\bigcap_{j\in J}\varphi_jN_0\varphi_j^{-1}\right)t^{-1}\right)
\end{equation*}
for any $n$ in $N$, $t$ in $T$, and $J\subseteq\{1,\dots,d\}$; and $V_x:=0$ if $\eta\neq b\eta_J$ for any $b$ in $P$ and $J\subseteq\{1,\dots,d\}$. The restriction maps are the natural inclusion maps. Indeed, for any two simplices $\eta_1\subseteq\eta_2$ such that $V_{\eta_2}\neq 0$ we have a $b=nt$ in $P$ such that $\eta_i=b\eta_{J_i}$ for $J_1\subseteq J_2\subseteq\{1,\dots,d\}$ and $i=1,2$ therefore $V_{\eta_2}=ntV_{\eta_{J_2}}$ is naturally contained in $V_{\eta_1}=ntV_{\eta_{J_1}}$ by extending the functions $f$ in $V_{nt\eta_J}$ to the whole $N$ by putting
$f_{\mid N\setminus nt\left(\bigcap_{j\in J}\varphi_jN_0\varphi_j^{-1}\right)t^{-1}}=0$. Later on we will often view elements of $V_{nt\eta_J}$ as functions on $N$ with support in $\mathrm{supp}(V_{nt\eta_J})=nt\left(\bigcap_{j\in J}\varphi_jN_0\varphi_j^{-1}\right)t^{-1}$.

Note that $V_{\eta}$ is either zero or equal to $\bigcap_{x\in\eta\cap\mathcal{BT}^0}V_x$. It might, however, happen that this intersection is nonzero but $V_{\eta}=0$ as $\eta$ is not in the $P$-orbit of $\eta_J$ for any $J\subseteq\{1,\dots,d\}$. We also see immediately that $P$ acts naturally on the coefficient system $(V_{\eta})$ and this action is compatible with the boundary maps. Moreover, we claim
\begin{lem}\label{e1}
We have
\begin{equation}\label{e12}
\bigoplus_{\eta\in\mathcal{BT}_q}V_{\eta}\cong\ind_{P_0Z}^P(V_q)
\end{equation}
with
\begin{equation*}
V_q:=\sum_{\begin{matrix}b_0\in P_0\text{, }|J|=q\\J\subseteq\{1,\dots,d\}\end{matrix}}V_{b_0\eta_J}=\bigoplus_{|J|=q,J\subseteq\{1,\dots,d\}}\bigoplus_{n_0\in N_0/\bigcap_{j\in J}\varphi_jN_0\varphi_j^{-1}}V_{n_0\eta_J}.
\end{equation*}
\end{lem}
\begin{proof}
By construction $V_q$ is a $P_0$-subrepresentation of $\bigoplus_{\eta\in\mathcal{BT}_q}V_{\eta}$ so we clearly have a $P$-equivariant map from the right hand side of \eqref{e12} to the left hand side. Since $V_q$ contains $V_{\eta_J}$ for any $q$-element subset $J$ of $\{1,\dots,d\}$ this map is surjective.

For the injectivity let $b$ be in $P$ with $b\eta_{J_1}=\eta_{J_2}$ for two (not necessarily distinct) subsets $J_1$ and $J_2$ of $\{1,\dots,d\}$. Assume that $b$ does not lie in $P_0Z$. Then we have $bx_0=\varphi_ix_0$ and $b\varphi_jx_0=x_0$ for some $1\leq i,j\leq d$. Hence $b=\varphi_ib_0$ for some $b_0$ in $\Stab_P(x_0)=P_0Z$ with $\varphi_ib_0\varphi_j$ lying also in $P_0Z$. This is a contradiction as $\varphi_ib_0\varphi_i^{-1}$ is in $P_0Z$, but $\varphi_i\varphi_j$ is not. It follows that $\eta_{J_1}$ and $\eta_{J_2}$ are in different $P$-orbits of $\mathcal{BT}$ if $J_1\neq J_2$ (since $\dim_{\mathbb{F}_p}L_{\varphi_ix_0}/pL_0=p^i$ for all $1\leq i\leq d$) and $\Stab_P(\eta_{J})\subseteq P_0Z$. The statement follows.
\end{proof}


\begin{lem}\label{e2}
The coefficient system $(V_\eta)_{\eta}$ defines an acyclic resolution of the representation $V_{St}$, ie.\ $H_0((V_{\eta})_{\eta})=V_{St}$ and $H_i((V_{\eta})_{\eta})=0$ for all $i\geq 1$.
\end{lem}
\begin{proof}
By Lemma \ref{e1} we note immediately that the natural map
\begin{equation}
\bigoplus_{\eta\in\mathcal{BT}_0}V_{\eta}\cong\ind_{P_0Z}^P(V_0)=\ind_{P_0Z}^P(M)\rightarrow V_{St}\label{e3}
\end{equation}
is surjective since $M$ generates $V_{St}$. On the other hand, if an element $f$ in $\ind_{P_0Z}^P(M)$ lies in the kernel of the above map \eqref{e3} then for some $t$ in $T_+$ the support of $tf$ lies in $P_+$. Hence for proving that $f$ lies in the image of $\bigoplus_{\eta\in\mathcal{BT}_1}V_{\eta}$ we may assume that $f$ has support in $P_+$. However, we claim that for any $b$ in $P_+$ and any element $v$ in $V_{bx_0}$ there is an element $v_0$ in $V_{x_0}$ such that $v-v_0$ lies in the image of $\bigoplus_{\eta\in\mathcal{BT}_1}V_{\eta}$. Indeed, if $b=n_0t$ for some $n_0$ in $N_0$ and $t$ in $T_+$ (since $P_+=N_0T_+$) then $v$ has support in $n_0tN_0t^{-1}\subseteq n_0t\varphi_j^{-1}N_0\varphi_jt^{-1}$ for any $j$ with $t\varphi_j^{-1}\in T_+$. Hence $v$ lies in $V_{\{n_0t\varphi^{-1}x_0,n_0tx_0\}}$ and the claim follows by induction on $K=\sum_{i=1}^dk_i$ with $tT_0Z=\prod_{i=1}^d\varphi_i^{k_i}T_0Z$. This shows that $H_0((V_{\eta})_{\eta})=V_{St}$.

For the acyclicity of the resolution $(V_\eta)_{\eta}$ we are going to
use Grosse-Kl\"onne's local criterion \cite{GK}. To recall his
result we need to introduce some terminology. Let $\hat{\eta}$ be
a pointed $(q-1)$-simplex with underlying $(q-1)$-simplex $\eta$.
Let $N_{\hat{\eta}}$ be the set of vertices $z$ of $\mathcal{BT}$
such that $(\hat{\eta},z)$ is a pointed $q$-simplex. Each
element $z$ in $N_{\hat{\eta}}$ corresponds to a lattice $L_z$
with $L_{q-1}\subsetneq L_z\subsetneq L_0$ where
$(L_0,\dots,L_{q-1})$ represents $\eta$. We call a subset $M_0$ of
$N_{\hat{\eta}}$ stable with respect to $\hat{\eta}$ if for any
two $z,z'$ in $M_0$ the lattice $L_z\cap L_{z'}$ represents an
element in $M_0$, as well. (By Lemma 2.2 in \cite{GK} this is
equivalent to the original definition of stability in the case of
the Bruhat-Tits building.) By Theorem 1.7 in \cite{GK} we need to
verify that for any $1\leq q\leq d$, any pointed $(q-1)$-simplex
$\hat{\eta}$, and any subset $M_0$ of $N_{\hat{\eta}}$ that is
stable with respect to $\hat{\eta}$ the sequence
\begin{equation}
\bigoplus_{\begin{matrix}z,z'\in M_0\\ \{z,z'\}\in\mathcal{BT}_1\end{matrix}}V_{\{z,z'\}\cup\eta}\rightarrow\bigoplus_{z\in M_0}V_{\{z\}\cup\eta}\rightarrow V_{\eta}\label{e5}
\end{equation}
is exact. Since our coefficient system is $P$-equivariant, we may
assume without loss of generality that $\eta=\eta_J$ for some
subset $J\subseteq\{1,\dots,d\}$ with $|J|=q-1$. Let $M_0\subseteq
N_{\hat{\eta_J}}$ be stable with respect to $\hat{\eta_J}$ (here
$\hat{\eta_J}$ corresponds to any fixed vertex of $\eta_J$). Since
the stabilizer of $\eta=\eta_J$ is contained in $P_0Z$, for any
simplex $\nu\supset\eta$ we have $\nu=n_\nu\eta_{J'}$ for some
$J'\supset J$ and $n_\nu$ in $N_0$ stabilizing $\eta$. In
particular, $\mathrm{supp}(V_{\nu})=n_\nu\left(\bigcap_{j\in
J'}\varphi_jN_0\varphi_j^{-1}\right)$. Hence for any $n_0$ in
$N_0$ the coset $n_0\bigcap_{j=1}^d\varphi_jN_0\varphi_j^{-1}$ is
either contained in $\mathrm{supp}(V_{\nu})$ or disjoint from
$\mathrm{supp}(V_{\nu})$. This means that we have
\begin{equation*}
V_\nu=C_c^{\infty}\left(n_{\nu}\bigcap_{j\in J'}\varphi_jN_0\varphi_j^{-1}\right)
=\bigoplus_{n_0\in n_{\nu}\bigcap_{i\in J'}\varphi_iN_0\varphi_i^{-1}/\bigcap_{j=1}^d\varphi_jN_0\varphi_j^{-1}}C_c^{\infty}\left(n_0\bigcap_{j=1}^d\varphi_jN_0\varphi_j^{-1}\right)
\end{equation*}
and it suffices to check the exactness of the restriction of \eqref{e5} to each coset $n_0\bigcap_{j=1}^d\varphi_jN_0\varphi_j^{-1}$. For any fixed $n_0$ we multiply the restriction of \eqref{e5} to the coset $n_0\bigcap_{j=1}^d\varphi_jN_0\varphi_j^{-1}$ by $n_0^{-1}$ and obtain the sequence
\begin{equation}\label{e6}
\bigoplus_{z\neq z'\in n_0^{-1}M_0\cap\{x_0,\dots,x_d\}}C_c^{\infty}(\bigcap_{j=1}^d\varphi_jN_0\varphi_j^{-1})\rightarrow\bigoplus_{z\in n_0^{-1}M_0\cap\{x_0,\dots,x_d\}}C_c^{\infty}(\bigcap_{j=1}^d\varphi_jN_0\varphi_j^{-1})\rightarrow C_c^{\infty}(\bigcap_{j=1}^d\varphi_jN_0\varphi_j^{-1})
\end{equation}
since the condition on $n_0$ lying in $n_{\nu}\bigcap_{i\in J'}\varphi_iN_0\varphi_i^{-1}/\bigcap_{j=1}^d\varphi_jN_0\varphi_j^{-1}$ is equivalent to that $n_0^{-1}\nu$ is a subsimplex of $\{x_0,\dots,x_d\}$. However, \eqref{e6} is clearly exact and the lemma follows.
\end{proof}

\begin{proof}[Proof of Proposition \ref{e16}]
At first we note that Lemma 11.8 in \cite{SV} generalizes to our case with the same proof, ie.\ $D(\ind_{P_0Z}^P(V))=D^0(\ind_{P_0Z}^P(V))$ and $D^i(\ind_{P_0Z}^P(V))=0$ for $i\geq 1$ for any smooth $P$-representation $V$ with central character since $Z\cong\mathbb{Q}_p^{\times}$ here, as well. So by Lemmas \ref{e1} and \ref{e2} (and noting that $Z$ acts trivially on each $V_q$) we may compute
\begin{equation*}
D^i(V_{St})=h^i(D(\bigoplus_{\eta\in\mathcal{BT}_{\bullet}}V_\eta)).
\end{equation*}
By Lemma 2.5 in \cite{SV} it suffices to show that for any $0\leq
q\leq d-1$ and any generating $P_+$-subrepresentation $M_{q+1}$ of
$\ind_{P_0Z}^P(V_{q+1})$ there exists a generating
$P_+$-subrepresentation $M_q$ of $\ind_{P_0Z}^P(V_q)$ such that
$M_q\cap\partial_{q+1}(\ind_{P_0Z}^P(V_{q+1}))\subseteq
\partial_{q+1}(M_{q+1})$. By (the analogue of) Lemma 3.2 in
\cite{SV} (see the proof of Lemma 11.8 in \cite{SV}) we may assume
that $M_{q+1}$ is of the form $M_{q+1}=M_{q+1,\sigma}$ for some
order reversing map $\sigma$ from $T_+/T_0Z$ to $\Sub(V_{q+1})$
satisfying
\begin{equation*}
\bigcup_{t\in T_+/T_0Z}\sigma(t)=V_{q+1}.
\end{equation*}
Here $\Sub(V_{q+1})$ denotes the partially ordered set of $P_0$-subrepresentations of $V_{q+1}$ and
\begin{equation*}
M_{q+1,\sigma}=\bigoplus_{t\in T_+/T_0Z}\ind_{P_0Z}^{N_0tP_0Z}\sigma(t)
\end{equation*}
where $\ind_{P_0Z}^X(V)$ denotes the set of functions with support in $X$ from $P$ to $V$ as a subset of $\ind_{P_0Z}^{P}(V)$ for any $P_0Z$-representation $V$ and $P_0Z$-invariant subset $X$ of $P$.

Moreover, since we have for any $n_0$ in $N_0$
\begin{equation*}
V_{n_0\eta_J}=C_c^{\infty}\left(n_0\bigcap_{j\in J}\varphi_jN_0\varphi_j^{-1}\right)=\bigcup_{n=0}^{\infty}C_c^{\infty}\left(n_0\bigcap_{j\in J}\varphi_jN_0\varphi_j^{-1}/\bigcap_{j'=1}^d\varphi_{j'}N_0^{p^n}\varphi_{j'}^{-1}\right)
\end{equation*}
with finite sets
\begin{equation*}
C_c^{\infty}\left(n_0\bigcap_{j\in J}\varphi_jN_0\varphi_j^{-1}/\bigcap_{j'=1}^d\varphi_{j'}N_0^{p^n}\varphi_{j'}^{-1}\right),
\end{equation*}
we may further assume (making $M_{q+1}$ possibly even smaller)
that $\sigma$ is induced by an unbounded order reversing map
$\sigma_0\colon T_+/T_0Z\rightarrow \mathbb{N}\cup \{-1\}$ with
\begin{equation*}
\sigma(t)=\sum_{\begin{matrix}n_0\in N_0\text{, }|J|=q+1\\ J\subseteq\{1,\dots,d\}\end{matrix}}V_{n_0\eta_J}(\sigma_0(t))
\end{equation*}
where
\begin{equation}
V_{n_0\eta_J}(\sigma_0(t)):=C_c^{\infty}\left(n_0\bigcap_{j\in J}\varphi_jN_0\varphi_j^{-1}/\bigcap_{j'=1}^d\varphi_{j'}N_0^{p^{\sigma_0(t)}}\varphi_{j'}^{-1}\right)\label{e7}
\end{equation}
for $\sigma_0(t)\geq 0$ and $V_{n_0\eta_J}(-1):=0$. Now we put
\begin{equation*}
M_q:=M_{q,\sigma_0}:=\bigoplus_{t\in T_+/T_0Z}\ind_{P_0Z}^{N_0tP_0Z}\sum_{\begin{matrix}n_0\in N_0\text{, }|J|=q\\ J\subseteq\{1,\dots,d\}\end{matrix}}V_{n_0\eta_J}(\sigma_0(t))
\end{equation*}
with $V_{n_0\eta_J}(\sigma_0(t))$ defined as in \eqref{e7}. We claim that
\begin{equation}
M_q\cap \partial_{q+1}(\ind_{P_0Z}^P(V_{q+1}))=\partial_{q+1}(M_{q+1}).\label{e8}
\end{equation}
We now distinguish two cases whether $q=0$ or bigger. In the case
$q>0$ the proof of \eqref{e8} is completely analogous to that of
Lemma \ref{e2}. We see by construction that
$\partial_{q+1}(M_{q+1})\subseteq M_q$. Hence we have the following
coefficient system on $\mathcal{BT}$ concentrated in degrees
$q+1$, $q$, and $q-1$. In degrees $q+1$ and $q$ we put $M_{q+1}$
and $M_q$, respectively as subspaces of
$\bigoplus_{\eta\in\mathcal{BT}_{q+1}}V_{\eta}=\ind_{P_0Z}^P(V_{q+1})$
and
$\bigoplus_{\eta\in\mathcal{BT}_q}V_{\eta}=\ind_{P_0Z}^P(V_q)$,
respectively. Indeed, we have by construction
\begin{eqnarray*}
M_{q+1}&=&\bigoplus_{\eta\in\mathcal{BT}_{q+1}}M_{q+1}\cap V_{\eta};\\
M_q&=&\bigoplus_{\eta\in\mathcal{BT}_q}M_q\cap V_{\eta}.
\end{eqnarray*}
In degree $q-1$ we put the whole $\ind_{P_0Z}^P(V_{q-1})$. We use
Grosse-Kl\"onne's criterion in order to show that the sequence
\begin{equation*}
M_{q+1}\rightarrow M_q\rightarrow \ind_{P_0Z}^P(V_{q-1})
\end{equation*}
is exact which implies \eqref{e8} as the kernel of the map from
$M_q$ to $\ind_{P_0Z}^P(V_{q-1})$ is exactly the left hand side of
\eqref{e8} by Lemma \ref{e2}. The proof proceeds the same way as
in Lemma \ref{e2}, but here all the functions are constant modulo
the subgroup
$\bigcap_{j=1}^d\varphi_jN_0^{p^{\sigma_0(t)}}\varphi_j^{-1}$
where $t$ only depends on $\eta$ (except for the case $\sigma_0(t)=-1$ whence all the functions are zero and the exactness is trivial). The sequence \eqref{e6}
remains exact if we replace $C_c^{\infty}(\bigcap_{j=1}^d\varphi_j
N_0\varphi_j^{-1})$ by
\begin{equation*}
C_c^{\infty}\left(\bigcap_{j=1}^d\varphi_j N_0\varphi_j^{-1}/\bigcap_{j=1}^d\varphi_j N_0^{p^{\sigma_0(t)}}\varphi_j^{-1}\right)
\end{equation*}
hence the statement.

For $q=0$ we have to be a bit more careful, since the inductional
argument in the proof of Lemma \ref{e2} does not work here as it
is not true that any $v$ in $M_0\cap V_{n_0tx_0}$ is equivalent to
some $v_0$ in $M_0\cap V_{x_0}$ modulo $\partial_{1}(M_{1})$. (Note that $M_0\cap V_{x_0}=V_{x_{0}}(\sigma_0(1))$ but  $M_0\cap V_{n_0tx_0}=V_{n_0tx_0}(\sigma_0(t))$ and $\sigma_0(t)$ could be much bigger than $\sigma_0(1)$.) However, we claim
that for any $v_t$ in $M_0\cap V_{n_0tx_0}$ with $n_0$ in $N_0$
and any $t'\leq t$ in $T_+$ there exists an element $v_{t'}$ in
\begin{equation*}
M_0\cap\left(\bigoplus_{n_1\in
N_0/t'N_0t'^{-1}}V_{n_1t'x_0}\right)
\end{equation*}
such that $v_t-v_{t'}$ lies in $\partial_1(M_1)$. The statement is
derived from this the following way. Any element $m$ in $M_0$ is
supported on finitely many vertices $\{b_it_ix_0\}_{i=1}^l$ of
$\mathcal{BT}$ with $t_i$ in $T_+$ and $b_i$ in $N_0$. Moreover,
there is a common $t'$ in $T_+$ with $t'\leq t_i$ for any $1\leq
i\leq l$. Now if $m$ lies in
$M_0\cap\partial_1(\ind_{P_0Z}^P(V_1))$ then by our claim there
exists an $m'$ in
\begin{equation}
M_0\cap\left(\bigoplus_{n_1\in
N_0/t'N_0t'^{-1}}V_{n_1t'x_0}\right)\label{e9}
\end{equation}
such that $m-m'$ lies in $\partial_1(M_1)$. However, the map from
\eqref{e9} to $V_{St}$ is injective since the supports of
functions in $V_{n_1t'x_0}$ and in $V_{n_1't'x_0}$ are disjoint
for $n_1n_1'^{-1}$ not in $t'N_0t'^{-1}$. It follows that $m'=0$
hence $m$ is in $\partial_1(M_1)$.

For the proof of the claim let $v_t$ be in $M_0\cap V_{n_0tx_0}$
for some $n_0$ in $N_0$ and $t$ in $T_+$. Then by definition of
$V_{n_0tx_0}$ the function $v_t$ is supported on
\begin{equation}
n_0tN_0t^{-1}=\bigcup^{\cdot}_{n_1\in
tN_0t^{-1}/t'N_0t'^{-1}}n_0n_1t'N_0t'^{-1}\label{e10}
\end{equation}
since $t'\leq t$ implies $t'N_0t'^{-1}\subseteq tN_0t^{-1}$.
Moreover, $v_t$ is constant on the cosets of
\begin{equation*}
t\left(\bigcap_{j=1}^{d}\varphi_jN_0^{p^{\sigma_0(t)}}\varphi_j^{-1}\right)t^{-1}
\end{equation*}
by the definition of $M_0$. We may assume by induction that
$t'=t\varphi_i$ for some $1\leq i\leq d$. Hence for any $n_1$ in
$tN_0t^{-1}/t\varphi_iN_0\varphi_i^{-1}t^{-1}$ the pair
$\{x_0,t^{-1}n_1t\varphi_ix_0\}$ represents an edge of
$\mathcal{BT}$. Therefore we have
\begin{equation}\label{e11}
M_1\cap
V_{\{n_0tx_0,n_0n_1t\varphi_ix_0\}}=C_c^{\infty}(n_0n_1(t\varphi_iN_0\varphi_i^{-1}t^{-1}/t\bigcap_{j=1}^d\varphi_jN_0^{p^{\sigma_0(t)}}\varphi_j^{-1}t^{-1}))
\end{equation}
and the map
\begin{equation*}
\pi_{n_0tx_0}\circ\partial_1\colon M_1\cap \left(\bigoplus_{n_1\in
tN_0t^{-1}/t'N_0t'^{-1}}V_{\{n_0tx_0,n_0n_1t\varphi_ix_0\}}\right)\rightarrow
M_0\cap V_{n_0tx_0}
\end{equation*}
is surjective comparing \eqref{e10} and \eqref{e11}. (Here
$\pi_{n_0tx_0}$ denotes the projection of $M_0$ onto $M_0\cap
V_{n_0tx_0}$.) The claim follows noting that
\begin{equation*}
\partial_1(M_1\cap
V_{\{n_0tx_0,n_0n_1t\varphi_ix_0\}})\subseteq
M_0\cap(V_{n_0tx_0}\oplus V_{n_0n_1t\varphi_ix_0}).
\end{equation*}
\end{proof}

The following is an immediate corollary of Remark 6.4 in \cite{SV} using Proposition \ref{e16}.

\begin{cor}
The natural transformation $a_V$ defined in section 6 of \cite{SV} gives an isomorphism
\begin{equation*}
a_{V_{St}}\colon V_{St}^*\rightarrow\psi^{-\infty}(D^0(V_{St})).
\end{equation*}
\end{cor}

\begin{rem}
Proposition \ref{e16} (and also Lemmas \ref{e1} and \ref{e2}) remain true in the following more general setting with basically the same proof. Let $G$ still be $\GL_{d+1}(\mathbb{Q}_p)$ and $V$ be a smooth $o$-torsion $P$-representation with a unique minimal generating $P_+$-subrepresentation $M$. Assume further that $nM\cap M=0$ for any $n$ in $N\setminus N_0$. Then we have $D^0(V)=D(V)$ and $D^i(V)=0$ for all $i\geq1$.
\end{rem}

\section*{Acknowledgement}

I gratefully acknowledge the financial support and the hospitality of the Max Planck Institute for Mathematics, Bonn. I would like to thank Peter Schneider for his constant interest in my work and for his valueable comments and suggestions. I would also like to thank Marie-France Vigneras for reading the manuscript and for her comments.

\begin{flushleft}Gergely Z\'abr\'adi\\
E\"otv\"os Lor\'and University, Mathematical Institute\\
Department of Algebra and Number Theory\\
Budapest\\
P\'azm\'any P\'eter s\'et\'any 1/C\\
H-1117\\
Hungary\\
zger@cs.elte.hu
\end{flushleft}
\end{document}